\documentclass{amsart}
\usepackage{amsmath,amsfonts,amssymb,amsthm}
\usepackage[utf8]{inputenc}
\usepackage[T1]{fontenc}
\usepackage{a4wide}
\usepackage{enumitem}
\usepackage{graphicx}
\usepackage{url}
\usepackage[dvipsnames]{xcolor}
\usepackage{comment}
\excludecomment{comment}
\usepackage{fourier}
\DeclareMathAlphabet{\mathcalzp}{OMS}{zplm}{m}{n}
\DeclareMathAlphabet{\mathdutchcal}{U}{dutchcal}{m}{n}
\SetMathAlphabet{\mathdutchcal}{bold}{U}{dutchcal}{b}{n}
\DeclareMathAlphabet{\mathdutchbcal}{U}{dutchcal}{b}{n}
\usepackage{tikz,tikz-cd}
\usetikzlibrary{arrows,arrows.meta}
\tikzcdset{arrow style=tikz, diagrams={>=latex}}
\usetikzlibrary{calc}

\specialcomment{commentDRK}{\begingroup \color{MidnightBlue}}{\endgroup}
\excludecomment{commentDRK}
\specialcomment{remarkDRK}{\begingroup \color{Aquamarine}\it }{\endgroup}
\excludecomment{commentDRK}

\newcommand{\cA}{\mathcal{A}}
\newcommand{\cB}{\mathcal{B}}
\newcommand{\BB}{\mathbf{B}}

\newcommand{\Eis}{\mathdutchbcal{Eis}}

\newcommand{\cM}{\mathdutchcal{M}}
\newcommand{\cO}{\mathcal{O}}

\newcommand{\cS}{\mathdutchbcal{S}}
\newcommand{\cX}{\mathcal{X}}

\newcommand{\EE}{\mathbf{E}}
\newcommand{\FF}{\mathbb{F}}
\newcommand{\GG}{\mathbb{G}}
\newcommand{\MM}{\mathbb{M}}
\newcommand{\NN}{\mathbb{N}}
\newcommand{\PP}{\mathbb{P}}
\newcommand{\QQ}{\mathbb{Q}}
\newcommand{\ZZ}{\mathbb{Z}}
\newcommand{\ZZhat}{\widehat{\ZZ}}
\newcommand{\XX}{\mathcalzp{X}\!}
\newcommand{\XC}{\mathcal{X}}

\newcommand{\GL}{\mathrm{GL}}
\newcommand{\SL}{\mathrm{SL}}

\renewcommand{\SS}{\mathrm{SS}}
\newcommand{\Aut}{\mathrm{Aut}}
\newcommand{\Isom}{\mathrm{Isom}}
\newcommand{\Jac}{\mathrm{Jac}}

\newcommand{\isom}{\cong}
\newcommand{\dsp}{\displaystyle}

\newcommand{\overbar}[1]{\mkern 2.0mu\overline{\mkern-2.0mu#1\mkern-2.0mu}\mkern 2.0mu}
\newcommand{\kbar}{\overbar{\smash{k}\vphantom{t}}}

\newcommand{\Graph}{\mathcalzp{G}}

\newcommand{\Group}{G}

\newcommand{\Hom}{\mathrm{Hom}}

\newcommand{\Gal}{\mathrm{Gal}}

\renewcommand{\SS}{\mathrm{SS}}

\newcommand{\tors}{\mathrm{tors}}

\newcommand{\mmu}{\boldsymbol{\mu}}
\newcommand{\ab}{\mathrm{ab}}
\newcommand{\cyc}{\mathrm{cyc}}

\newtheorem{theorem}{Theorem}

\newtheorem{proposition}[theorem]{Proposition}
\newtheorem{definition}[theorem]{Definition}
\newtheorem*{theorem*}{Theorem}
\newtheorem*{definition*}{Definition}

\makeatletter
\newcommand{\leqnomode}{\tagsleft@true\let\veqno\@@leqno}
\newcommand{\reqnomode}{\tagsleft@false\let\veqno\@@eqno}

\makeatother

\makeatletter
\newcommand*{\relrelbarsep}{.264ex}
\newcommand*{\relrelbar}{%
  \mathrel{%
    \mathpalette\@relrelbar\relrelbarsep
  }%
}
\newcommand*{\@relrelbar}[2]{%
  \raise#2\hbox to 0pt{$\m@th#1\relbar$\hss}%
  \lower#2\hbox{$\m@th#1\relbar$}%
}
\providecommand*{\rightrightarrowsfill@}{%
  \arrowfill@\relrelbar\relrelbar\rightrightarrows
}
\providecommand*{\leftleftarrowsfill@}{%
  \arrowfill@\leftleftarrows\relrelbar\relrelbar
}
\providecommand*{\xrightrightarrows}[2][]{%
  \ext@arrow 0359\rightrightarrowsfill@{#1}{#2}%
}
\providecommand*{\xleftleftarrows}[2][]{%
  \ext@arrow 3095\leftleftarrowsfill@{#1}{#2}%
}
\makeatother

\begin{document}

\title{Supersingular isogeny graphs and Hecke modules with level structure}
\author{Leonardo Col\`o and David Kohel}
\date{}
\reqnomode

\begin{abstract}
  We study supersingular isogeny graphs with level structure and their associated
  Galois representations.
\end{abstract}

\maketitle

\section{Introduction}

The study of supersingular isogeny graphs has seen recent renewed interest
for their applications to cryptography.  This has resulted in advances in
algorithms for their construction and study.  The principal focus is on
explicit algorithms for local transversal of an $\ell$-isogeny graph,
beginning at a given supersingular point, under the hypothesis that the
characteristic $p$ is large.  The global properties of the isogeny
graphs, in particular the Ramanujan property, assures that random walks
give good mixing properties and that short cycles are rare.

On the other hand, in the prior work of Mestre and Oesterl\'e \cite{Mestre86},
and of Pizer \cite{Pizer80} in the quaternion ideal setting, the motivation
for studying isogeny graphs stems from the application to construction of
modular forms and their associated Galois representation.  The equivalence
of the quaternion and supersingular approaches builds on the Deuring
correspondence \cite{Deuring} and associated equivalence of categories
\cite{Kohel96}.  In this setting, one needs to study the global properties
of the graphs, such that the supersingular points, and associated $\ell$-isogeny
graphs, can be enumerated.  Consequently the prime $p$ remains small.
Nevertheless the computation tools for cryptographic construction apply
also to the investigation of the associated Galois representations,.

The method of graphs of Mestre interpreets the adjacency operators of
$\ell$-isogeny graphs as correspondences defining the Hecke operator
$T_\ell$ acting on the basis of supersingular points.
Mestre restricts to supersingular points, over $\FF_{p^2}$, on $X(1)$
or one of the genus-$0$ modular curves $X_0(N)$ of prime level, for
$N \in \{2,3,5,7,13\}$.  In this setting, the use of explicit correspondences
is en effective tool for computing $\ell$-isogenies, but becomes cumbersome
for moderate $\ell$.  One goal of the present work is to generalize the
modular approach to curves of higher level.  This has several advantages
and applications.  First we are abel to extend beyond semistable Galois
representation to study additive reduction.  Secondly, for a degree-$d$
cover $\XC_G \to X(1)$ by a modular curve $\XC_G$ of genus-$0$, the
size of the modular correspondences on $\XC_G$ are typically reduced by
a factor of $d$.  Combined with a sparseness of monomials, dictated by
certain combinatorial restructions, which gives an additional constant
factor to the reduction, the result can be spectacular.  Finally, by
pulling back $\ell$-isogeny graphs by such a cover, the computational
advantages of these reductions can be applied in lower level, in
particular level~$1$.

\medskip
\begin{commentDRK}
  After refining the final form we should detail the structure and results
  section by section.
\end{commentDRK}

\section{Graphs with level structure}
\label{Section:ModularGraphs}

We recall that an isogeny graph $\Graph_S(E)$ of an elliptic curve
$E/k$ and finite set of primes $S$ is a directed graph whose vertices
are elliptic curves $E_i$, up to $\kbar$-isomorphism, which are
$\kbar$-isogenous to $E$, and whose directed edges from a representative
curve $E_1$ are isogenies $E_1 \to E_2$ of prime degree $\ell \in S$, up
to isomorphism of the codomain $E_2$.  In particular there
are $\ell+1$ outgoing edges of degree $\ell$ from each vertex $E_i$.
When $S = \{\ell\}$, we will write simply $\Graph_\ell(E)$ for the graph.
In this section we generalize these notions to elliptic curves with
arbitrary level structure. For this purpose we adopt, the conventions
of Sutherland and Zywina~\cite{SutherlandZywina} for modular curves $\XC_G$
defined by open subgroups $\Group$ of $\GL_2(\ZZhat)$.
Morevoer, for given $G$ of level $N$, defined by its image in $\GL_2(\ZZ/N\ZZ)$,
we describe a category of pairs $(E,\cB\Group)$ where $E/k$ is an elliptic
curve and $\cB = (P,Q)$ is a basis of $N$-torsion on $E$.

\subsection*{Automorphism groups of torsion modules}

Let $\ZZhat$ be the projective limit of $\ZZ/N\ZZ$ and let $\GL_n(\ZZhat)$
be the general linear group over $\ZZhat$, equipped with projections
$\pi_N : \GL_n(\ZZhat) \to \GL_n(\ZZ/N\ZZ)$ for every $N$.
We say that a subgroup $G$ of $\GL_2(\ZZhat)$ is {\it open}
if and only if there exists $N$ such that
$
G = \pi_N^{-1}(\pi_N(G)),
$
and the minimal such $N$ is called the level of $G$.

Let $E/k$ be an elliptic curve over a field $k$. A choice of
compatible bases for the torsion subgroups $E[N] \subset E(\kbar)$
gives an isomorphism of the Tate module:
$$
T(E) = \varprojlim_{N} E[N] \isom \ZZhat^2 = \varprojlim_{N} (\ZZ/N\ZZ)^2,
$$
with automorphism group $\Aut(T(E)) \isom \GL_2(\ZZhat)$, equipped
with projections $$
\Aut(T(E)) \to \Aut(E[N]) \isom \GL_2(\ZZ/N\ZZ).
$$
The quotient $T(E) \to E[N]$ induces $\Aut(T(E)) \to \Aut(E[N])$,
which induces the natural projections
$$
\pi_N:\GL_2(\ZZhat) \longrightarrow
\GL_2(\ZZ/N\ZZ),
$$
under the choice of basis for $T(E)$, defined as the limit of
compatible bases for $E[N]$.

Analogously, let $\GG_m$ be the multiplicative group over $k$ and $\mmu_N$ its
subgroup of $N$-torsion in $\GG_m(\bar{k})$.  Let $\mmu_\infty$ be the injective
limit of the $\mmu_N$, and let $T(\GG_m)$ be the Tate module of $\GG_m$, defined
as the projective limit
$$
T(\GG_m) = \varprojlim_{N} \mmu_N.
$$
Although $T(\GG_m)$ is isomorphic to $\ZZhat$, with $\Aut(T(\GG_m) \isom \ZZhat^*$,
we write the group law of $T(\GG_m)$ multiplicatively. For $\zeta \in T(\GG_m)$ and
$n \in \ZZhat^*$, we write $\zeta^n$ for the image  of $(n,\zeta)$ under the action
$$
\Aut(T(\GG_m)) \times T(\GG_m) \to T(\GG_m).
$$
The Weil pairing gives an alternating Galois equivariant pairing
$$
\hat{e}: T(E) \times T(E) \longrightarrow T(\GG_m)
$$
which is compatible with the determinant map, in the sense that for
$x, y \in T(E)$ and $\sigma \in \Gal(\bar{k}/k)$, 
$$
\hat{e}(x,y)^\sigma =
\hat{e}(x^\sigma,y^\sigma) =
\hat{e}(x,y)^{\det(\rho_E(\sigma))}.
$$

\noindent{\bf Remark.}
The groups $(\ZZ/N\ZZ)^n$ are equipped with both systems of injections and surjections
$$
\begin{tikzcd}[row sep=0pt]
  (\ZZ/N\ZZ)^n \arrow[r] & (\ZZ/\ell{N}\ZZ)^n\\
  (1,\dots,1) \arrow[r,mapsto] & (\ell,\dots,\ell),
\end{tikzcd}
\quad
\begin{tikzcd}[row sep=0pt]
  (\ZZ/N\ZZ)^n & \arrow[l] (\ZZ/\ell{N}\ZZ)^n\\
  (1,\dots,1)  & \arrow[l,mapsto] (1,\dots,1),
\end{tikzcd}
$$
such that the injective and projective limits 
$$
\varinjlim_N (\ZZ/N\ZZ)^n \isom (\QQ/\ZZ)^n \mbox{ and }
\varprojlim_N (\ZZ/N\ZZ)^n \isom \ZZhat^n
$$
are equipped, respectively, with injections of  $(\ZZ/N\ZZ)^n$ and surjections on $(\ZZ/N\ZZ)^n$.
An automorphism $\sigma$ in $\Aut(\ZZhat^n)$ or $\Aut((\QQ/\ZZ)^n)$ induces a
compatible system of automorphisms in $\Aut((\ZZ/N\ZZ)^n)$, making the
diagram commute:
$$
\begin{tikzcd}
  \ZZhat^n \arrow[d] \arrow[r,"\sigma"] &  \ZZhat^n \arrow[d]\\
  (\ZZ/N\ZZ)^n \arrow[d,hook] \arrow[r] &  (\ZZ/N\ZZ)^n \arrow[d,hook]\\
  (\QQ/\ZZ)^n \arrow[r,"\sigma"] &   (\QQ/\ZZ)^n
\end{tikzcd}
$$
Conversely a unique automorphism is induced by a compatible system of automorphisms
of $(\ZZ/N\ZZ)^n$. Consequently, we have canonical isomorphisms
$
\GL_n(\ZZhat) \isom
\Aut(\ZZhat^n) \isom 
\Aut((\QQ/\ZZ)^n)$.  Any Galois representations compatible with the
system of injections and surjections hence injects in $\GL_n(\ZZhat)$,
which can be identified with automorphisms of either $(\QQ/\ZZ)^n$ or
$\ZZhat^n$.  In the context of elliptic curves, this means we identify
$$
\GL_2(\ZZhat) \isom \Aut(T(E)) \isom \Aut(E_\tors),
$$
and in the context of the multiplicative group, we have
$$
\GL_1(\ZZhat) = \ZZhat^* \isom \Aut(T(\GG_m)) \isom \Aut(\mmu_\infty).
$$

\subsection*{Group actions on bases}

Suppose that $\Group \subset \GL_2(\ZZhat)$ is an open subgroup of
level~$N$ and $E/k$ is an elliptic curve over a field $k$ of characteristic
coprime to $N$. Let $\cB = (P,Q)$ be an ordered basis for $E[N] \subset E(\kbar)$.
We denote by $\EE$ the pair $(E,\cB\Group)$, where $\cB\Group$ is
the orbit of $\cB$ under the right action of $\Group$ given by
$$
(P,Q)\left(\begin{array}{@{}cc@{}}a&b\\c&d\end{array}\right) = (aP+cQ,bP+dQ),
$$
and we refer to $\EE$ as an enhanced elliptic curve.
  
Let $\BB(E,N)$ be the set of all bases $\cB = (P,Q)$ for $E[N]$,
equipped with a right action of $\GL_2(\ZZ/N\ZZ)$,
$$
\BB(E,N) \times \GL_2(\ZZ/N\ZZ) \longrightarrow \BB(E,N).
$$
Then $\BB(E,N)$ is a principal homogeneous space over $\GL_2(\ZZ/N\ZZ)$,
and we obtain a map:
$$
\begin{tikzcd}[row sep=0mm]
  \BB(E,N) \times \GL_2(\ZZ/N\ZZ) \times (\ZZ/N\ZZ)^2 \arrow[r] & E[N].\\
  (\cB,\gamma,v) =
  \left(
    (P,Q),
    \left(\begin{array}{@{}c@{\;}c@{\;}c@{}}a&b\\c&d\end{array}\right),
    \left[\begin{array}{@{}c@{}}x\\y\end{array}\right]
    \right)
  \arrow[r,mapsto] & \cB \gamma v = (ax + by)P + (cx + dy)Q
\end{tikzcd}
$$
Each fixed $\cB$ in $\BB(E,N)$ equips $E[N]$ with an isomorphism
$(\ZZ/N\ZZ)^2 \isom E[N]$ and consequent left $\GL_2(\ZZ/N\ZZ)$-action.
In particular, there is an induced bijection:
$$
\begin{tikzcd}[row sep=0mm]
  \iota: \BB(E,N) \arrow[r] & \Isom((\ZZ/N\ZZ)^2,E[N])\\
  (P,Q) \arrow[r,mapsto] & 
    \left(\left[\begin{array}{@{}c@{}}x\\y\end{array}\right] \mapsto xP + yQ\right)
\end{tikzcd}
$$
which carries the right $\GL_2(\ZZ/N\ZZ)$-action on $\BB(E,N)$ to the
left $\GL_2(\ZZ/N\ZZ)$-action on $\Isom((\ZZ/N\ZZ)^2,E[N])$.
For $\cB \in \BB(E,N)$, we have
$\gamma^* \iota(\cB) = \iota(\cB) \circ \gamma = \iota(\cB \gamma)$,
and settting $\phi_\cB = \iota(\cB) \in \Isom(\ZZ/N\ZZ)^2,E[N])$, we have
$$
(\gamma_1\gamma_2)^*\phi_\cB\left(\left[\begin{array}{@{}c@{}}x\\y\end{array}\right]\right)
= {\gamma_2}^*{\gamma_1}^*\phi_\cB\left(\left[\begin{array}{@{}c@{}}x\\y\end{array}\right]\right)
= {\gamma_1}^*\phi_\cB\left(\gamma_2\left[\begin{array}{@{}c@{}}x\\y\end{array}\right]\right)
= \phi_\cB\left(\gamma_1\gamma_2\left[\begin{array}{@{}c@{}}x\\y\end{array}\right]\right)\cdot
$$
In working with bases we obtain a right action on bases $\BB(E,N)$ rather than a left
action on isomorphisms $\Isom((\ZZ/N\ZZ)^2,E[N])$.

Let $e_N: E[N] \times E[N] \longrightarrow \kbar^*$ the Weil pairing.  For any $\cB = (P,Q)
\in \BB(E,N)$, the image $e_N(\cB) = e_N(P,Q)$ is a primitive $N$-th root of unity in $\kbar^*$.
This gives a well-defined map $e_N: \BB(E,N) \longrightarrow \mmu_N$ whose image is the
subset of primitive $N$-th roots of unity.
It follows from the bilinearity and alternating properties of $e_N$ that for all $\gamma
\in \GL_2(\ZZhat)$, we have
$$
e_N(\cB\gamma) = e_N(\cB)^{\det(\gamma)}, \mbox{ and hence } e_N(\cB{G}) = e_N(\cB)^{\det(G)}.
$$
By the Galois equivariance of the Weil pairing, the compositum $\det \circ \rho_E$ agrees
with the restriction map $G_k \rightarrow \Gal(k(\zeta_N)/k)$.

\subsubsection*{Systems of bases and orbits}
A sequence $(\cB_N)_N = ((P_N,Q_N))_N$ of bases for $E[N]$ is a projective system
if for all positive integers $M$ and $N$, $[M]\cB_{MN} = (MP_{MN},MQ_{MN}) = (P_N,Q_N) = \cB_N$.
The projective system determines the projective limit,
$$
\hat{\cB} = \varprojlim_{N} \cB_N,
$$
as a $\ZZhat$-basis for the Tate module $T(E)$, such that the projections
$\pi_N: T(E) \rightarrow E[N]$ induce $\pi_N(\hat{\cB}) = \cB_N$.
Conversely every such basis $\hat{\cB}$ uniquely determines the projective system of bases.
We denote the set of all $\ZZhat$-bases for $T(E)$ by $\BB(T(E))$, equipped with
induced projections $\pi_N: \BB(T(E)) \rightarrow \BB(E,N)$.

\begin{proposition}
  Let $G$ be an open subgroup of $\GL_2(\ZZhat)$ of level $N$.
  For $M \in \NN$, let $\pi_M: \GL_2(\ZZhat) \rightarrow \GL_2(\ZZ/M\ZZ)$
  be the reduction map, and set $G_M = \pi_M^{-1}(\pi_M(G))$.

  \begin{itemize}
  \item
    If $N_1$ and $N_2$ are coprime integers such that $N = N_1N_2$, then
    the projection maps determine a bijection of orbits:
    $$
    [N_2] \times [N_1]:
    \BB(E,N)/G \longrightarrow
    \BB(E,N_1)/G_{N_1} \times \BB(E,N_2)/G_{N_2}.
    $$
  \item
    For every positive integer~$M$, there is a bijection
    $$
    [M]: \BB(E,MN)/G \longrightarrow \BB(E,N)/G,
    $$
    taking $\cB_{MN} \in \BB(E,MN)$ to $[M]\cB_{MN} \in \BB(E,N)$.
    \smallskip
  \item
    There exists a bijection between $\BB(T(E))/G$ and $\BB(E,N)/G$.
  \end{itemize}
\end{proposition}

\smallskip
\noindent{\bf Remark.}
The previous proposition asserts that the action of $G$ is local and that the orbits are
determined by the quotient to the $N$-torsion.

\subsection*{Graphs with level structure}
For each given elliptic curve $E$, basis $\cB$ for $E[N]$, and open subgroup $G$
of level $N$, there exists a finite set of classes:
$$
\cB_1 G = \cB\gamma_1G,\
\cB_2 G = \cB\gamma_2G, \dots,
\cB_n G = \cB\gamma_nG, \mbox{ such that } \bigcup_{i=1}^n \cB \gamma_i G = \BB(E,N),
$$
in bijection with $\GL_2(\ZZ/N\ZZ)/G$. This gives a finite number $\EE_1$, $\EE_2$,
\dots $\EE_n$ of classes over each $E$.  We first define an isogeny graph whose
vertices are enhanced elliptic curves, which thereafter can be identified wth points
a modular curve $\XC_G$.

Let $\EE = (E,\cB\Group)$ be a fixed enhanced elliptic curve of level $N$.
For a set $S$ of primes coprime to $N$, the isogeny graph $\Graph_S(\Group,\EE)$
is the graph whose vertices are pairs
$\EE_i = (E_i,\cB_i\Group)$ consisting of an elliptic curve
$E_i$ in the $\kbar$-isogeny class of $E$ and an ordered
basis $\cB_i = (P_i,Q_i))$ for $E_i[N]$, and whose edges are
isogenies $\varphi: E_1 \longrightarrow E_2$ of prime degree $\ell \in S$,
such that the equivalence class
$\varphi\big(\cB_1\Group\big) = \varphi(\cB_1)\Group = \cB_2\Group$.
As for graphs of level~$1$, when $S = \{\ell\}$, we will write simply
$\Graph_\ell(\Group,\EE)$.
We denote $\SS(\Group,p)$ the set of pairs $\EE = (E,\cB\Group)$
where $E$ is supersingular, and denote the associated graph of supersingular curves
with level structure by $\Graph_S(\SS(\Group,p))$.

For an inclusion $H \subseteq \Group$ of open groups, setting
$\EE{G}  = (E,\cB{H})G = (E,\cB{G})$ we obtain a covering projection of graphs:
$$
\Graph_S(H,\EE) \longrightarrow \Graph_S(\Group,\EE{G})
$$
sending the vertex $(E_i,\cB_i H)$ to $(E_i,\cB_i\Group)$, and
an edge $\varphi: E_1 \longrightarrow E_2$ of $\Graph(H,\EE)$ to an edge
of $\Graph(\Group,\EE{G})$, since if $\varphi(\cB_1)H = \cB_2H$ then
$\varphi(\cB_1)\Group =  \cB_2\Group$.
In particular, associated to the inclusion $\Group \subseteq \GL_2(\ZZhat)$
we obtain the covering $\Graph_S(\Group,\EE) \longrightarrow \Graph_S(E)$
of the level~$1$ graph.

The description in terms of enhanced elliptic curves $\EE = (E,\cB{G})$ implicitly
requires enumerating orbits of bases over the splitting field for the $N$-torsion
subgroup $E[N]$.  However, when working with the subgroup of the form $B_1(N)$
the enhanced elliptic curve $\EE$ can be represented by a pair $(E,P)$, for
a Borel subgroup $B_0(N)$ the curve $\EE$ is represented by a pair
$(E,\langle{P}\rangle)$, where the subgroup $\langle{P}\rangle$ is specified by
the kernel polynomial $\psi(x)$ such that $\psi(x(aP)) = 0$ for all $a \in (\ZZ/N\ZZ)^*$.
Next we describe the approach via modular curves, in order to define
modular isogeny graphs, in which we identify enhanced elliptic curves with
points on the modular curve $\XC_G$.

\section*{Modular groups and modular curves}


\subsection*{Galois representations}
For a field $k$, we write $k^{\cyc}$ for the extension $k(\mmu_\infty) \subset \bar{k}$:
$$
k^{\cyc} = k(\mmu_\infty) = \varinjlim_N k(\mmu_N) \mbox{ where } \mmu_\infty = \varinjlim_N \mmu_N.
$$
In particular for $k = \QQ$, the field $\QQ^{\cyc} = \QQ^{\ab}$ is the maximal
abelian extension of $\QQ$ by the Kronecker--Weber theorem.
In general we identify $\Gal(k^\cyc/k)$ with a subgroup of $\ZZhat^* \isom \Aut(\mmu_\infty)
\isom \Aut(T(\GG_m))$.

The action of the Galois group on the torsion subgroups induces a Galois representations:
$$
\rho_{E,N}: G_k \longrightarrow \GL_2(\ZZ/N\ZZ) \isom \Aut(E[N]),
$$
extending to the projective limit:
$$
\rho_E: G_k = \Gal(\kbar/k) \longrightarrow \GL_2(\ZZhat) \isom \Aut(T(E)).
$$
When $k$ is a number field, Serre's open image theorem~\cite{Serre72} asserts that if
the curve $E/k$ is non-CM, then the image $\rho_E(G_k)$ is open, and in particular of
finite index.

The projections $\pi_N:\GL_2(\ZZhat) \longrightarrow \GL_2(\ZZ/N\ZZ)$
are compatible with the representations the $N$-torsion subgroups of $E$,
in the sense that $\rho_{E,N} = \pi_N \circ \rho_E$.
Moreover, the composition with determinant map $\det \circ \rho_E$ gives
the cyclotomic representation of $G_k$ restricted to $k^\cyc$. 

\subsection*{Admissible groups}
Let $G \subseteq \GL_2(\ZZhat)$ be an open subgroup of level $N$.  The objective
is to define the notion of {\it admissible} group $G$ such that there exists
smooth proper modular curve $\XC_G$ over $\ZZ[1/N]$, whose points are identified
with enhanced elliptic curves $\EE = (E,\cB\Group)$.  We require first a condition
for $\XC_G$ to be defined over $\QQ$, over a number field $K \subseteq \QQ(\zeta_N)$,
or over a finite quotient field $k$ of $\cO_K[1/N]$.

For a number field $K/\QQ$, we identify
$$
\Gal(K(\zeta_N)/K) \subseteq \Gal(\QQ(\zeta_N)/\QQ) = (\ZZ/N\ZZ)^*.
$$
The group $G$ is said to be admisssible over $K$ if $\det(\pi_N(G))$ contains
$\Gal(K(\zeta_N)/K$, and when $K = \QQ$ we say simply that $G$ is admissible
if $\det(\pi_N(G)) = (\ZZ/N\ZZ)^*$.  Admissibility gives a necessary condition
for the modular curve $\XC_G$ to be defined over $K$, in particular that the
orbit
$$
e_N(\cB{G}) =
\{ e_N(\cB\gamma) : \gamma \in G \} =
\{ e_N(\cB)^{\det(\gamma)} : \gamma \in G \} = e_N(\cB)^{\det(G)}
$$
is stable under $\Gal(K(\zeta_N)/K)$. While the orbit depends on
the choice of $\cB$ (and $E$), the stability condition does not.
Conversely the admissibility of $G$ is sufficient to define $\XC_G$
over $K$.
and when $K = \QQ$, hence $\det(\pi_N(G)) = (\ZZ/N\ZZ)^*$, the curve
can be defined over $\ZZ[1/N]$. 
Given a congruence subgroup $\Gamma \subseteq \SL_2(\ZZ)$, there may be multiple
lifts to $G \subseteq \GL_2(\ZZhat)$ such that $\Gamma = \SL_2(\ZZ) \cap G$,
which gives ambiguity regarding the twist associated to $\Gamma$ in descending
from $\QQ(\zeta_N)$ to $\QQ$.
By working with open subgroups $G \subset \GL_2(\ZZhat)$ we avoid this
ambiguity and retain a closer correspondence with the computational model
of orbits $\cB{G}$ of $N$-torsion points on supersingular elliptic curves.

\subsection*{Modular curves}
\mbox{}
Each of the standard congruences subgroups $\Gamma$ equal to $\Gamma(N)$, $\Gamma_1(N)$, $\Gamma_0(N)$
or to one of the Cartan subgroups $\Gamma_s(N)$ or $\Gamma_{ns}(N)$ has an admissible lift to
$G \subseteq \GL_2(\ZZhat)$.  In particular, we define the lifts
$$
G(N)   = \pi_N^{-1}\left(\left\{\left(\begin{array}{@{\,}c@{\;\;}c@{\,}} \pm1 & 0 \\ 0 & {*} \end{array}\right)\right\}\right)
\mbox{ such that } \Gamma(N) = \SL_2(\ZZ) \cap G(N),
$$
parametrizing elliptic curves $E$ equipped with an isomorphism $\ZZ/N\ZZ \times \mmu_N \isom
\langle{P,Q}\rangle = E[N]$, as group schemes,
$$
B_1(N) = \pi_N^{-1}\left(\left\{\left(\begin{array}{@{\,}c@{\;\;}c@{\,}} \pm1 & {*} \\ 0 & {*} \end{array}\right)\right\}\right)
\mbox{ such that } \Gamma_1(N) = \SL_2(\ZZ) \cap B_1(N),
$$
parametrizing elliptic curves $E$ with a constant group scheme $\ZZ/N\ZZ \isom \langle{P}\rangle \subset E[N]$, and
$$
B_0(N) = \pi_N^{-1}\left(\left\{\left(\begin{array}{@{\,}c@{\;\;}c@{\,}} {*} & {*} \\ 0 & {*} \end{array}\right)\right\}\right)
\mbox{ such that } \Gamma_0(N) = \SL_2(\ZZ) \cap B_0(N),
$$
parametrizing elliptic curves with a cyclic subgroup $\langle{P}\rangle$ of order $N$.
We denote the respective modular curves $\XC_G$ by $X(N)$, $X_1(N)$ and $X_0(N)$.

\begin{proposition}
  Let $\Group$ be an open subgroup in $\GL_2(\ZZhat)$ of level $N$ and
  $(E,\cB\Group)$ a pair consisting of an elliptic curve $E/K$ and the
  orbit of a basis $\cB \in \BB(E,N)$.
  Then $(E,\cB\Group)$ is associated to a rational point
  in $\cX_G(K)$ if and only if
  $$
  \phi_\cB^{-1}\rho_{E,N}(G_K)\phi_\cB \subset \pi_N(G).
  $$
  where $\phi_\cB: (\ZZ/N\ZZ)^2 \to E[N]$ is the isomorphism induced by $\cB$.
\end{proposition}

A point on $X(N)$ over a number field $K$ can be identified with an elliptic curve
$E/K$ with Galois action of $\rho_{E,N}(G_K) \subseteq \phi_\cB \pi_N(G(N)) \phi_\cB^{-1}$,
with respect to a basis $\cB = (P,Q)$ for $E[N]$.
This implies the existence of isomorphisms $\langle{P}\rangle \isom \ZZ/N\ZZ$,
taking $P \in E(K)$ to $1 \in \ZZ/N\ZZ$, and $\langle{Q}\rangle \isom \mmu_N$ taking
$Q$ to $e_N(P,Q) = \zeta_N \in \mmu_N$, compatible with the action of Galois,
which maps through
$$
\det(G) = \Gal(K(\zeta_N)/K) \subseteq \det(G(N)) = (\ZZ/N\ZZ)^*.
$$
After taking the quotient by the Borel subgroup $B_1(N)$, an enhanced elliptic curve
$(E,\cB B_1(N))$ can be identified with the pair $(E,\pm P)$, which is associated to
a point on $X_1(N)$.  Finally the points on the curve $X_0(N)$ depend only on the
pair $(E,\langle{P}\rangle)$ consisting of a curve and Galois-stable subgroup
$\langle{P}\rangle$ without prescribed generator.

The modular curves $X_0(N)$ are of particular note, equipped with an Atkin-Lehner involution
$$
w_N: X_0(N) \longrightarrow X_0(N)
$$
such that $w_N^2 = 1$. Specifically, given an elliptic curve $E/K$ with basis $\cB = (P,Q)$
such that $\rho_{E,N}(G_K) \subseteq \phi_\cB \pi_N(B_0(N)) \phi_\cB^{-1}$,
the enhanced elliptic curve $(E,\cB B_0(N))$
is determined by the tuple $(E,\langle{P}\rangle)$ to which we associate a $K$-rational
point.  The Atkin-Lehner involution is determined on points by the map
$$
(E,\langle{P}\rangle) \longmapsto ((E/\langle{P}\rangle),(E[N]/\langle{P}\rangle)).
$$
Composing the canonical projection $\pi: X_0(N) \rightarrow X(1)$, sending
$(E,\langle{P}\rangle)$ to $E$, with $w_N$, sends $(E,\langle{P}\rangle)$ to
$E/\langle{P}\rangle$.
We denote $X_0(N)$ equipped with this pair of maps by:
$$
X_0(N) \xrightrightarrows[\hspace{5mm}]{\hspace{5mm}} X(1),
$$
which is equivalent to the data of an immersion $X_0(N) \rightarrow X(1) \times X(1)$,
a correspondence on the surface $X(1) \times X(1)$. These maps are fundamental to the
definition of Hecke operators as correspondences.

We now turn to the definition of the Cartan modular curves.
Let $R$ be an imaginary quadratic ring with optimal embedding $\iota: R \to \MM_2(\ZZ)$,
determining $\iota_N : R \to \MM_2(\ZZ/N\ZZ)$.  We define the Cartan subgroup $C_\iota(N)$
of level $N$ associated to $\iota$ by
$$
C_\iota(N) = \pi_N^{-1}\left(\iota_N(R)^*\right).
$$
If every primes divisor $p$ of $N$ splits in $R$ we say that $C_\iota(N)$ is the split
Cartan subgroup of level $N$, denoted $C_s(N)$, and conversely if every prime divisor
is inert in $R$, we say that $C_\iota(N)$ is the nonsplit Cartan subgroup of level $N$,
denoted $C_{ns}(N)$. The split or inert Cartan subgroups are unique up to conjugation.
As for the other classical modular curves, we denote the associated modular curves
$\XC_G$ by $X_s(N)$ or $X_{ns}(N)$ respectively, or more generally by $X_\iota(N)$.
The Cartan subgroups admit involutions by conjugation at each $p$ dividing $N$,
and we denote the normalizer subgroups of $\GL_2(\ZZhat)$ by $C_s^+(N)$, $C_{ns}^+(N)$
or $C_\iota^+(N)$, with associated modular curves $X_{s}^+(N)$, $X_{ns}^+(N)$
or $X_\iota^+(N)$.
\bigskip

For an inclusion of admissible groups $H \subset G$, of levels $M$ and $N$, we
obtain a morphisms $\XC_{H} \to \XC_{G}$ over $\ZZ[1/M]$.
In particular, for $G = \GL_2(\ZZhat)$
we obtain the $j$-line, $X(1) = \PP^1/\ZZ$, equipped with $\XC_H \rightarrow X(1)$
for all open subgroups $H \subset \GL_2(\ZZhat)$.   
Given arbitrary open subgroups $H$ and $G$ in $\GL_2(\ZZhat)$, we denote by
$\XC_G(H)$ the modular curve associated to $G \cap H$, equipped with the cover
$\XC_G(H) \longrightarrow \XC_G$.  In particular $\XC_G(B_0(\ell))$, for $\ell$
coprime to the level of $G$, gives the correspondence
$$
\XC_G(B_0(\ell)) \xrightrightarrows[\hspace{5mm}]{\hspace{5mm}} \XC_G.
$$
\subsection*{Finite base fields}
The previous discussion of open subgroups is framed in terms of a number field $K$,
introducing the condition of admissibility to justify when the field of definition
of the modular curves descends to $K$. However, when considering the specialization
to finite fields, especially $\FF_{p^2}$, the condition for admissibility is simpler:
we just need
$$
\langle{p^2}\rangle \subset \det(G) \subset (\ZZ/N\ZZ)^*.
$$
When $N = 24$, this is automatically satisfied (for $p > 3$), since $p^2 \equiv 1 \bmod 24$. 

For the study of supersingular points, it suffices to work over
$\ZZ[\zeta_N,1/N]$, and a quotient field $k$, and identify $\XC_G$ with a classical
modular curve $\XC_\Gamma$, where $\Gamma$ is a congruence subgroup of $\SL_2(\ZZ)$.
The graph vertices are identified with supersingular points in $\XC_G(k)$. 
The generalization to an open subgroup $G$ in $\GL_2(\ZZhat)$ permits one to
control the twists of $\XC_\Gamma$, and fits better with the computational model
in which we represent vertices as enhanced elliptic curves, modulo the action of
a subgroup of $G$.

\subsection*{Modular isogeny graphs}

We can identify the vertices of $\Graph_S(\Group,\EE)$ with points on the modular curve $\XC_\Group$.
An enhanced elliptic curve $\EE = (E,\cB\Group)$, with $E/k$ and such that $\rho_{E}(\Group_k)
\subseteq G$, is an associated $k$-rational point on the modular curve~$\XC_G$.
An edge of $\Graph_S(\Group,\EE)$ is associated with a point on the modular curve $\XC_G(B_0(\ell))$,
for $\ell$ coprime to $N$, 
and otherwise $\XC_G(B_0(\ell^t))$, where $t$ is the smallest exponent such that $G$
is not contained in $B_0(\ell^t)$.
The correspondence $\XC_G(B_0(\ell^t)) \rightrightarrows \XC_G$ gives the initial
and terminal vertices of the edge.
When emphasizing the perspective of moduli points on the modular curve $\XC_G$, we
write $\Graph_S(\XC_G,\EE)$ for the modular isogeny graph associated to $G$ whose
vertices are rational points on $\XC_G$.
Similarly, we write $\Graph_S(\SS(\XC_\Group,p))$ for the associated supersingular isogeny
graph on the set $\SS(\XC_\Group,p)$, of supersingular points on $\XC_G$.

\subsection*{Independence and hybrid level structures}
Next we introduce the notion of independence of level structures with a view to defining
isogeny graphs using hybrid models of elliptic curves parametrized by modular curves and
equivalence classes of torsion points.

\begin{definition}
Given open subgroups $H_1$ and $H_2$ of $\GL_n(\ZZhat)$ of levels $N_1$
and $N_2$, set $G = \langle{H_1,H_2}\rangle$ and $H = H_1 \cap H_2$
of level $N = \mathrm{lcm}(N_1,N_2)$. We say that $H_1$ and $H_2$ determine
an $H$-structure of level $N$, and say that they are {\it independent} in $G$
over $H$ if the following equivalent conditions are satisfied:
$$
[G:H] = [G:H_1][G:H_2], \mbox{ or }
[G:H_1] = [H_2:H], \mbox{ or }
[G:H_2] = [H_1:H].
$$
We say that $H_1$ and $H_2$ are geometrically independent if the groups
$H_1 \cap \SL_n(\ZZhat)$ and $H_2 \cap \SL_2(\ZZhat)$ are independent in
$G \cap \SL_n(\ZZhat)$ over $H \cap \SL_n(\ZZhat)$.
\end{definition}

\noindent{\bf Remark.}
The equivalence of the conditions for independence follows from the equalities:
$$
[G:H] = [G:H_1][H_1:H] = [G:H_2][H_2:H].
$$

\noindent
The notion of independence corresponds to the equality of degrees
$$
\begin{array}{c}
m_1 = \deg\left(\cX_H\to\cX_{H_1}\right) = \deg\left(\cX_{H_2}\to\cX_G\right) = n_2,\\[2mm]
m_2 = \deg\left(\cX_H\to\cX_{H_2}\right) = \deg\left(\cX_{H_1}\to\cX_G\right) = n_1,
\end{array}
$$
in the commutative diagram of morphisms of modular curves:
$$
\begin{tikzcd}[column sep=3mm,row sep=6mm]
  & \cX_H
  \arrow[dl,"m_1",swap,start anchor={[xshift=+1mm,yshift=+1mm]south west}, end anchor={[xshift=+1mm]north}]
  \arrow[dr,"m_2",start anchor={[xshift=-1mm,yshift=+1mm]south east}, end anchor={[xshift=-1mm]north}] \\
  \cX_{H_1}
  \arrow[dr,"n_1",swap,start anchor={[xshift=-1mm,yshift=+1mm]south east}, end anchor={[xshift=-1mm]north}]
  & & \cX_{H_2}
  \arrow[dl,"n_2",start anchor={[xshift=+1mm,yshift=+1mm]south west}, end anchor={[xshift=+1mm]north}] \\
  & \cX_G
\end{tikzcd}
$$
A trivial instance of independence occurs when $N_1$ and $N_2$ are coprime, for which
$G = \GL_2(\ZZhat)$.  This allows one to decompose a level structure into independent
level structures of prime-power levels $N = \ell^n$.

\begin{proposition}
  Let $H_1$ and $H_2$ be open supgroups of $G = \GL_2(\ZZhat)$. Any two of following
  conditions implies the third.
  \begin{enumerate} 
    \item The subgroups $H_1$ and $H_2$ are independent in $G$.
    \item The subgroups $H_1$ and $H_2$ are geometrically independent in $G$.
    \item The subgroups $\det(H_1)$ and $\det(H_2)$ are independent in $\ZZhat^*$.
  \end{enumerate} 
\end{proposition}

\begin{proof}
  For any open subgroup $G \subset \GL_2(\ZZhat)$ of level $N$,
  and reduction map $\pi_N: \GL_2(\ZZhat) \to \GL_2(\ZZ/N\ZZ)$,
  we set $G_0 = \pi_N(G)$ and $G_1 = G_0 \cap \SL_2(\ZZ/N\ZZ)$.
  From the reduction $\bmod N$ of the exact sequence,
  $$
  \begin{tikzcd}
    1 \arrow[r] &
    \SL_2(\ZZhat) \cap G \arrow[r] \arrow[d] &
    G \arrow[r] \arrow[d] &
    \det(G) \arrow[r] \arrow[d] & 1 \\
    1 \arrow[r] & G_1 \arrow[r] & G_0 \arrow[r] & \det(G_0) \arrow[r] & 1
  \end{tikzcd}
  $$
  and the identity $[\GL_2(\ZZhat):G] = [\GL_2(\ZZ/N\ZZ):G_0]$, the
  multiplicative relation
  $$
  [\GL_2(\ZZ/N\ZZ):G_0] = [\SL_2(\ZZ/N\ZZ):G_1][(\ZZ/N\ZZ)^*:\det(G_0)].
  $$
  gives the required dependency relation between the three independence conditions.
\end{proof}

\noindent{\bf Example.}
  The Borel subgroup $B_0(2)$ and the nonsplit Cartan subgroup $C_{ns}(2)$
  are independent and geometrically independent over $H = G(2)$.
  This corresponds to the diagram of curves:
  $$
  
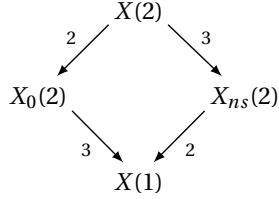
\begin{tikzcd}[column sep=3mm,row sep=6mm]
    & X(2)
    \arrow[dl,"2",swap,start anchor={[xshift=+1mm,yshift=+1mm]south west}, end anchor={[xshift=+2.5mm]north}]
    \arrow[dr,"3",start anchor={[xshift=-1mm,yshift=+1mm]south east}, end anchor={[xshift=-3mm]north}] \\
    X_0(2)
    \arrow[dr,"3",swap,start anchor={[xshift=-1mm,yshift=+1mm]south east}, end anchor={[xshift=-2mm]north}]    & & X_{ns}(2)
    \arrow[dl,"2",start anchor={[xshift=+1mm,yshift=+1mm]south west}, end anchor={[xshift=+2mm]north}] \\
    & X(1)
  \end{tikzcd}
  $$
  Since the cyclotomic representions in $(\ZZ/2\ZZ)^* = \{1\}$ are trivial, independence
  and geometric independence are equivalent.

\begin{proposition}
  The Borel subgroup $B_0(2^n)$ and the nonsplit Cartan subgroup
  $C_{ns}(2^n)$ are independent, but not geometrically independent
  for $n > 1$.
\end{proposition}


\noindent{\bf Example.}
For a prime $N = \ell^n$, $\ell$ an odd prime, the subgroups
$H_1 = B_0(N)$ and $H_2 = C_{ns}^+(N)$ are independent in
$G = \GL_2(\ZZhat)$ over the intersection
$$
H = B_0(N) \cap C_{ns}^+(N) =
    (\ZZ/N\ZZ)^*I_2
    \cdot
    \left\langle
    \left(\begin{array}{@{\;}cc@{\;}}
      -1 & 0 \\ 0 & 1
    \end{array}\right)
    \right\rangle
    \cdot
  $$
  On the other hand, 
  $$
  \det(\pi_N(B_0(N))) = \det(\pi_N(C_{ns}^+(N))) = \det(\pi_N(G)) = (\ZZ/N\ZZ)^*,
  $$
  while $\det(\pi_N(H)) = \{\pm1\}((\ZZ/N\ZZ)^*)^2$.  It follows that
  $B_0(N)$ and $C_{ns}^+(N)$ are geometrically independent if
  and only if $\ell \equiv 3 \bmod 4$.
  This gives the following result.

\begin{proposition}
  Let $N = \ell^n$ for an odd prime $\ell$.
  The subgroups $B_0(N)$ and $C_{ns}^+(N)$ are independent, and
  geometrically independent if and only if $\ell \bmod 3 \bmod 4$.
\end{proposition}

A decomposition of level structure into pairwise independent subgroups $(H_1,\dots,H_t)$,
with $H = \bigcap_i H_i$ permits one to decompose the $H$-level structure into a hybrid
combination of $H_i$-orbits of torsion points or rational points on $\cX_{H_i}$.
\bigskip


\section{Explicit Isogeny graphs}

We describe several examples which illustrate the covering morphisms and new structures
obtained from the isogeny graphs with level structure.
\medskip

\noindent{\bf Example.}
Let $d$ be a squarefree integer, $D$ the discriminant of $\QQ(\sqrt{d})$,
and $\chi_D$ the associated quadratic character. The nonsplit Cartan modular
curve $X_{ns}(2)$ is defined by the cover $j(u) = u^2 + 1728$ of the
$j$-line $X(1)$, and we define the twisted nonsplit Cartan curve
$X_{ns}^d(2)$ by the cover $j(u) = u^2/d + 1728$.
Sutherland and Zywina~\cite[Remark~3.4]{SutherlandZywina} define the
associated open subgroup of $\GL_2(\ZZhat)$ as follows.
Let $\varepsilon: \GL_2(\ZZhat) \to \{\pm1\}$ be the unique quadratic
character which maps through $\GL_2(\FF_2)$, whose kernel is the
nonsplit Cartan subgroup $C_{ns}(2) = \ker(\varepsilon)$ of level~$2$.
The twist by the quadratic character $\chi_D \circ \det: \GL_2(\ZZhat)
\to \{\pm1\}$, gives a twisted nonsplit Cartan subgroup:
$$
C_{ns}^d(2) = \{ \gamma \in \GL_2(\ZZhat) \;:\; \chi_D(\det(\gamma)) \cdot \varepsilon(\gamma) = 1\}.
$$
The placement of $d$ in the denominator plays the role of the
(squarefree part) of the discriminant, since $j(q) = {E_6^2(q)}/{\Delta(q)} + 1728$,
and $E_6(q)$ and $\Delta(q)$ have polynomial expressions $c_6(E)$ and
$\Delta(E)$ in the coefficients of a given curve~$E/K$.
Consequently, a parametrization, $j(E) = c_6(E)^2/\Delta(E) + 12^3$, holds
if and only if $\Delta(E) \equiv d \bmod (K^*)^2$.  This condition is
satisfied, on the other hand, if and only if $K(\sqrt{d}) \subset K(E[2])$.

We can now describe the covering graphs associated to $X_{ns}^d(2) \to X(1)$,
and show that for different $d$, the graphs are indeed distinguished.
For each prime $\ell$ coprime to $2d$, we obtain a map
$$
\Graph_\ell(\SS(X_{ns}^{d}(2),p)) \longrightarrow \Graph_\ell(\SS(X(1),p)),
$$
which is a double cover on vertices, away from the ramified point
$j = 12^3$.  Over the field $\FF_{11^2} = \FF_{11}[i]$, with $i^2 = -1$,
we consider the respective supersingular $3$-isogeny graphs for $d = 1$
and $d = -1$ in Figure~\ref{figure:Cartan_pair_p=11}.

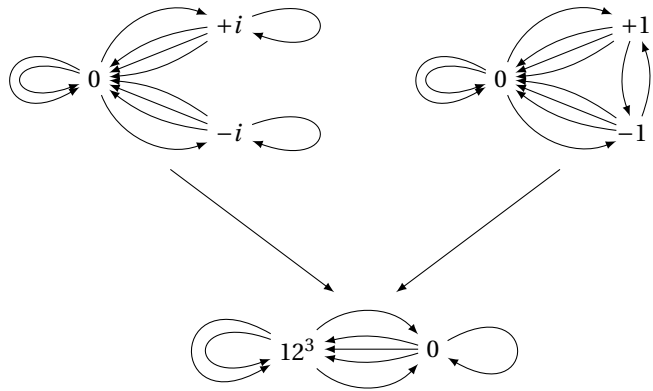
\begin{figure}[b]
$$
\begin{tikzpicture}[scale=0.90]
  \node (G1) at (1.000,-1.250) {$$}; 
  \node (d0) at (0.000,0.000) {$0$};
  \node (d1) at (2.000,+0.800) {$+i$};
  \node (d2) at (2.000,-0.800) {$-i$};
  \draw[->, >=latex] (d0) to [out=160, in=200, looseness=20] (d0);
  \draw[->, >=latex] (d0) to [out=150, in=210, looseness=16] (d0);
  \draw[->, >=latex] (d0) to [bend left=45] (d1);
  \draw[->, >=latex] (d0) to [bend right=45] (d2);
  \draw[->, >=latex] (d1) to (d0);
  \draw[->, >=latex] (d1) to [bend left=15] (d0);
  \draw[->, >=latex] (d1) to [bend right=15] (d0);
  \draw[->, >=latex] (d2) to (d0);
  \draw[->, >=latex] (d2) to [bend left=15] (d0);
  \draw[->, >=latex] (d2) to [bend right=15] (d0);
  \draw[->, >=latex] (d1) to [out=20, in=340, looseness=15] (d1);
  \draw[->, >=latex] (d2) to [out=20, in=340, looseness=15] (d2);
  \node (G2) at (7.000,-1.250) {$$}; 
  \node (c0) at (6.000,0.000) {$0$};
  \node (c1) at (8.000,+0.800) {$+1$};
  \node (c2) at (8.000,-0.800) {$\!\!\!-1$};
  \draw[->, >=latex] (c0) to [out=160, in=200, looseness=20] (c0);
  \draw[->, >=latex] (c0) to [out=150, in=210, looseness=16] (c0);
  \draw[->, >=latex] (c0) to [bend left=45] (c1);
  \draw[->, >=latex] (c0) to [bend right=45] (c2);
  \draw[->, >=latex] (c1) to (c0);
  \draw[->, >=latex] (c1) to [bend left=15] (c0);
  \draw[->, >=latex] (c1) to [bend right=15] (c0);
  \draw[->, >=latex] (c2) to (c0);
  \draw[->, >=latex] (c2) to [bend left=15] (c0);
  \draw[->, >=latex] (c2) to [bend right=15] (c0);
  \draw[->, >=latex] (c1) to [bend right=20] (c2);
  \draw[->, >=latex] (c2) to [bend right=20] (c1);
  \node (G0_L) at (3.666,-3.250) {$$}; 
  \node (G0_R) at (4.333,-3.250) {$$}; 
  \node (b0) at (3.000,-4.000) {$12^3$}; 
  \node (b1) at (5.000,-4.000) {$0$}; 
  \draw[->, >=latex] (b0) to [out=160, in=200, looseness=12] (b0);
  \draw[->, >=latex] (b0) to [out=150, in=210, looseness=10] (b0);
  \draw[->, >=latex] (b0) to [bend left=45] (b1);
  \draw[->, >=latex] (b0) to [bend right=45] (b1);
  \draw[->, >=latex] (b1) to (b0);
  \draw[->, >=latex] (b1) to [bend right=15] (b0);
  \draw[->, >=latex] (b1) to [bend left=15] (b0);
  \draw[->, >=latex] (b1) to [out=30, in=330, looseness=16] (b1);
  \draw[->, >=latex] (G1) to (G0_L);
  \draw[->, >=latex] (G2) to (G0_R);
\end{tikzpicture}
$$
\caption{Supersingular $2$-isogeny twisted Cartan graphs covers for $X_{ns}^t(2) \to X(1)/\FF_{11^2}$}
\label{figure:Cartan_pair_p=11}
\end{figure}

\noindent
The respective adjacency matrices are
$$
T_3 = \left(\begin{array}{@{}ccc@{}}2 & 1 & 1\\ 3 & 1 & 0\\ 3 & 0 & 1\end{array}\right) \mbox{ and }
T_3' = \left(\begin{array}{@{}ccc@{}}2 & 1 & 1\\ 3 & 0 & 1\\ 3 & 1 & 0\end{array}\right).
$$
We note that in general the automorphism of the supersingular points, exchanging
the points over each point of $\SS(X(1),p)$, stabilizes the graphs.  For this
example, this automorphism is given by the matrix
$$
U = \left(\begin{array}{@{}ccc@{}}1 & 0 & 0\\ 0 & 0 & 1\\ 0 & 1 & 0\end{array}\right),
$$
and we see that $U$ commutes with $T_3$ and $T_3' = UT_3 = T_3U$.  More generally
the automorphism $U$ commutes with all Hecke operators $T_\ell$ and either
$T_\ell' = U T_\ell$ if $\chi_D(\ell) = -1$ and$T_\ell' = T_\ell$ if $\chi_D(\ell) = 1$.
\bigskip

\noindent{\bf Example.}
Consider the diagonal subgroup $G = \{\pm\mathrm{diag}(1,d): d \in \ZZ/5\ZZ^*\}$
of $\GL_2(\ZZ/5\ZZ)$, whose intersection with $\SL_2(\ZZ/5\ZZ)$ is
$\overline{\Gamma(5)} = \{\pm I\}$.
The associated modular curve $\XC_G/\ZZ[1/5] = X(5)/\ZZ[1/5]$ has genus~$0$,
equipped with an $A_5$-Galois cover $X(5) \to X(1)$ of the $j$-line,
defined by the map $u \mapsto j(u)$, where
$$
j(u) = \frac{(u^{20} + 228u^{15} + 494u^{10} - 228u^5 + 1)^3}{u^5(u^{10} - 11u^5 - 1)^5}\cdot
$$
The supersingular point $j = 0 = 12^3$ on $X(1)/\FF_2$ corresponds to an elliptic
curve $E$ with $|\Aut(E)| = 24$, and splits into the five supersingular points
$\{1,\zeta_5,\zeta_5^2,\zeta_5^3,\zeta_5^4\}$ on $X(5)$ over $\FF_2[\zeta_5] = \FF_{2^4}$,
each with multiplicity~$12$.
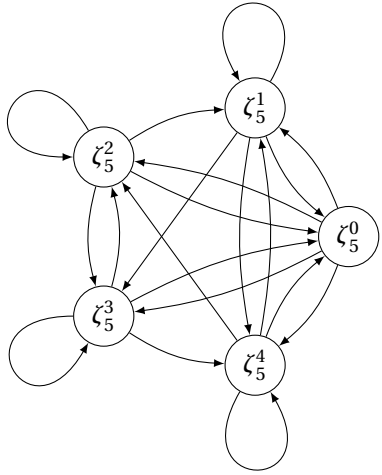
\begin{figure}
$$
\begin{tikzpicture}[scale=0.90]
  \node[draw, circle] (c0) at (+2.000,+0.000) {$\zeta_5^0$};
  \node[draw, circle] (c1) at (+0.618,+1.902) {$\zeta_5^1$};
  \node[draw, circle] (c2) at (-1.618,+1.176) {$\zeta_5^2$};
  \node[draw, circle] (c3) at (-1.618,-1.176) {$\zeta_5^3$};
  \node[draw, circle] (c4) at (+0.618,-1.902) {$\zeta_5^4$};
  \draw[->, >=latex] (c0) to [bend right=15] (c1);
  \draw[->, >=latex] (c0) to [bend right=10] (c2);
  \draw[->, >=latex] (c0) to [bend left=10] (c3);
  \draw[->, >=latex] (c0) to [bend left=15] (c4);
  \draw[->, >=latex] (c1) to [out=60, in=120, looseness=10] (c1);
  \draw[->, >=latex] (c1) to [bend right=15] (c0);
  \draw[->, >=latex] (c1) to (c3); 
  \draw[->, >=latex] (c1) to [bend right=10] (c4);
  \draw[->, >=latex] (c2) to [out=120, in=180, looseness=10] (c2);
  \draw[->, >=latex] (c2) to [bend right=10] (c0);
  \draw[->, >=latex] (c2) to [bend left=15] (c1);
  \draw[->, >=latex] (c2) to [bend right=15] (c3);
  \draw[->, >=latex] (c3) to [out=180, in=240, looseness=10] (c3);
  \draw[->, >=latex] (c3) to [bend left=10] (c0);
  \draw[->, >=latex] (c3) to [bend right=15] (c2);
  \draw[->, >=latex] (c3) to [bend right=15] (c4);
  \draw[->, >=latex] (c4) to [out=240, in=300, looseness=10] (c4);
  \draw[->, >=latex] (c4) to [bend right=10] (c1);
  \draw[->, >=latex] (c4) to (c2); 
  \draw[->, >=latex] (c4) to [bend left=15] (c0);
\end{tikzpicture}
$$
\caption{Supersingular $3$-isogeny graph of level $5$ over $\FF_{2^2}$}
\label{figure:Level_5_p=2}
\end{figure}
The graph of $3$-isogenies is given in Figure~\ref{figure:Level_5_p=2}, with adjacency matrix
$$
T_3 = 
\left(
\begin{array}{ccccc}
0 & 1 & 1 & 1 & 1\\
1 & 1 & 0 & 1 & 1\\
1 & 1 & 1 & 1 & 0\\
1 & 0 & 1 & 1 & 1\\
1 & 1 & 1 & 0 & 1
\end{array}
\right)\cdot
$$
The role of these adjacency matrices as Hecke operators, acting on the free
abelian group of supersingular points, is the object of the application of
correspondences on modular curves and graph adjacency to analysis of Galois
representations.

\section{Hecke modules on supersingular points and monodromy}


\subsection*{Supersingular modules with level structure}
Let $p$ be a prime, $G$ be an open subgroup of $\GL_2(\hat{\ZZ})$ of level $M$
coprime to $p$, and $\cS_G = \SS(\cX_G,p))$ the set of supersingular points on
$\cX_G$ over an algebraic closure $\bar{\FF}_p$ an of $\FF_p$. The supersingular
points divisor module on $\cX_G$ is the free abelian group
$$
\cM(\cS_G) = \bigoplus_{\EE\in\cS_G} \ZZ\cdot[\EE],
$$
and denote its $\cX(\cS_G)$ the subgroup of degree~$0$ divisors. Here we identify
an isomorphism class $[\EE]$ of enhanced elliptic curve with its associated point
on $\cX_G$.  For each $n$ coprime to $Mp$, we define the Hecke operators $T_n$
acting on $\cM(\cS_G)$ by
$$
T_n([\EE]) = \sum_{\varphi_i} [\EE_i],
$$
where the sum is over cyclic isogenies $\varphi_i: \EE \longrightarrow \EE_i$
of degree $n$, up to isomorphism of the codomain curve $\EE_i$.

For enhanced elliptic curves $\EE_1$ and $\EE_2$, let $\Isom(\EE_1,\EE_2)$ be the
set of isomorphisms from $\EE_1$ to $\EE_2$.  We define an inner product on
$\cM(\cS_G)$ by
$$
\langle[\EE_1],[\EE_2]\rangle = \frac{1}{2} |\Isom(\EE_1,\EE_2)|,
$$
extending $\langle\,,\,\rangle$ bilinearly to $\cM(\cS_G) \times \cM(\cS_G)$.
The Hecke operators are Hermitian with respect to the inner product:
$$
\langle[\EE_1],T_n[\EE_2]\rangle = \langle T_n[\EE_1],[\EE_2]\rangle.
$$
The orthogonal complement to $\cX(\cS_G)$ is the rank one submodule generated
over $\QQ$ by the element
$$
\Eis = \sum_{[E]\in\cS_G} \frac{1}{\langle[\EE],[\EE]\rangle} [\EE] \in \cM(\cS_G) \otimes \QQ,
$$
which we call the Eisenstein subspace of $\cM(\cS_G)$.

\subsection*{New subspaces}
Suppose that $H$ is an open subgroup of $\GL_2(\hat{\ZZ})$ containing $G$, and
$\pi: \cX_G \longrightarrow \cX_H$ the associated morphism of modular curves.
This morphism induces pushforward and pullback maps:
$$
\begin{tikzcd}[row sep=-2mm]
  \pi_* : \cX(\cS_G) \arrow[r] & \cX(\cS_H) & & \pi^* : \cX(\cS_H) \arrow[r] & \cX(\cS_G) \\
     {[\EE]} \arrow[r,mapsto]       & {[\pi(\EE)]} & & {[\EE]} \arrow[r,mapsto]       & \dsp \sum_{i = 1}^{\deg(\pi)}[\EE_i]
\end{tikzcd}
$$
where $\EE_i$ runs over the elements of $\pi^{-1}(\EE)$, with multiplicities.
The kernel $\ker(\pi_*)$ is called the $H$-new subspace of $\cX(\cS_G)$, and
the pullback $\pi^*(\cX(\cS_H))$ is called the $H$-old subspace of $\cX(\cS_G)$.
The intersection of the $H$-new subspaces for $H$ running over all open
subgroups $H$ (minimally) containing $G$ is called the new subspace
of $\cX(\cX_G)$ and the submodule generated by the $H$-old subspaces is called
the old subspace of $\cX(\cX_G)$.

One easily verifies that the $H$-old and $H$-new subspaces of $\cX(\cS_G)$
are orthogonal with respect to the inner product, and stabilized by the Hecke
operators $T_n$ with $n$ coprime to $N$, and consequently the old and new
subspaces are Hecke-invariant orthogonal submodules.
\bigskip


\noindent{\bf Remark.}
Mestre remarks that this construction, for Borel subgroups, is equivalent to
a divisor modular on left quaternion ideals of a quaternion Eichler order
described by Pizer (following Brandt and Eichler), and that this module can
be identified with a subspace of modular forms.
The equivalence of categories between supersingular elliptic curves and
left ideals of a maximal quaternion order~(see
Deuring~\cite[\S10.2]{Deuring} and Kohel~\cite[\S5.3]{Kohel96})
implies that the construction extends functorially to level structures
defined by open subgroups of $\GL_2(\hat{\ZZ})$.
\bigskip

\noindent{\bf Remark.}
The monodromy group $\cX(\cA,p)$ at $p$ of an abelian variety $\cA/\QQ$ with
semistable reduction at $p$ is the character group the toric part $T$ of
the special fiber at $p$ of its Neron model:
$$
\cX(\cA,p) = \Hom_{\bar{\FF}_p}(T,\GG_m).
$$
With $G$ and $p$ as above, set $G_0(p)$ be the intersection of $G$ with the
Borel subgroup $B_0(p)$ in $\GL_2(\hat{\ZZ})$. For $\cA = \Jac(\cX_{G_0(p)})$,
Grothendieck~\cite{SGA7} proves that the monodromy group can be canonically
identified with the supersingular divisor group:
$$
\cX(\cA,p) = \cX(\cS_G).
$$


\section*{Sieving for elliptic curves}

Supersingular modules permit one to compute the modular forms (or Galois
representations) associated to elliptic curves or modular abelian varieties
of given conductor.  Cowan~\cite{Cowan22} (reference) uses Mestre's original
construction~\cite{Mestre86} to sieve for low dimensional modular abelian
varieties of prime conductor. We illustrate the analogous construction
using the supersingular modules with level structure, which permits us
to determine empirical distributions of modular abelian varieties, with
prescribed ramification at small primes, beyond those readily accessible
in databases.

The rank of the supersingular modules $\XX(p,G)$ grows linearly with $p$,
but the Brandt matrices, determining the Hecke operators on $\XX(p,G)$,
are sparse.  Following Cowan, in order to study existence of elliptic
factors in $\XX(p,G)$, it suffices to sieve for the kernels:
$$
  \ker(T_\ell-c), \ker(T_\ell-c+1), \dots \ker(T_\ell+c-1), \ker(T_\ell+c), 
$$
where $c = \lfloor 2\sqrt{\ell}\rfloor$ is the Hasse-Weil bound.
This allows one to study existence of elliptic curves with semistable
reduction at $p$ and reduction type dictated by $G$ at primes dividing $N$.
Due to the cofactor of $p$ in the level, even for moderate $p$ we
rapidly exceed the levels in standard databases.

\bigskip

\noindent{\bf Example.}
The curves parametrized by Weber functions give genus~$0$ modular curves
of level 48.  As an example, we exhibit an orbit of twists of elliptic
curves which appear in the supersingular module of a Weber curve of
characteristic $p = 3851$.  The initial traces of Frobenius are given
in the table below.

\noindent
\begin{minipage}{0.75\textwidth}
{\small
\begin{verbatim}
Modular form traces of Frobenius:
[   5   7  11  13  17  19  23  29  31  37  41  43  47  53  59  61  67  71  73  79  83  89]
[   3   0   1   4   0  -5   3  -3   8   2   6  -8   9   2  12  -3  11  -8 -11  -4  -4 -15]
[   3   0  -1   4   0   5  -3  -3  -8   2   6   8  -9   2 -12  -3 -11   8 -11   4   4 -15]
[  -3   0   1   4   0   5   3   3  -8   2  -6   8   9  -2  12  -3 -11  -8 -11   4  -4  15]
[  -3   0  -1   4   0  -5  -3   3   8   2  -6  -8  -9  -2 -12  -3  11   8 -11  -4   4  15]
\end{verbatim}
}
\end{minipage}
\bigskip

\noindent
This shows the existence of modular elliptic curves of conductor $1,109,088 = 3851 \cdot 288$,
with additive reduction at 2 and 3 and multiplicative reduction at $p = 3851$.
Mention that these curves are not currently in the publically accessible LMFDB database~\cite{LMFDB},
but can be found in incomplete databases of higher conductor curves.


\end{document}